\documentclass[12pt,twoside]{amsart}

\usepackage{hyperref}
\usepackage{amsthm, amsmath, amscd, amssymb,centernot}
\usepackage[all]{xy}
\usepackage[T1]{fontenc}
\usepackage[left=2.5cm,top=2.5cm,bottom=3cm,right=2.5cm]{geometry}

\newcommand{\ints}{\mathbb Z}

\newcommand{\str}{\mathcal{O}}

\newcommand{\proj}{\mathbb{P}}

\newcommand{\complex}{\mathbb C}

\theoremstyle{plain}
\numberwithin{equation}{section}
\newtheorem{theorem}{Theorem}[section]
\newtheorem*{theorem*}{Theorem}

\newtheorem*{conjecture*}{Nagata Conjecture}
\newtheorem*{conjecture1*}{SHGH Conjecture}
\theoremstyle{definition}
\newtheorem{question}[theorem]{Question}

\newtheorem{remark}[theorem]{Remark}
\newtheorem{example}[theorem]{Example}

\begin{document}
\title{Seshadri constants on surfaces with Picard number 1}
\author[Krishna Hanumanthu]{Krishna Hanumanthu}
\thanks{Author was partially supported by a grant from Infosys Foundation}
\address{Chennai Mathematical Institute, H1 SIPCOT IT Park, Siruseri,
  Kelambakkam 603103, India}
\date{October 17, 2016}
\email{krishna@cmi.ac.in}
\maketitle
\begin{abstract}
Let $X$ be a smooth projective surface with Picard number 1. Let $L$ be
the ample generator of the N\'eron-Severi group of $X$. Given an
integer $r\ge 2$, we prove lower bounds
for the Seshadri constant of $L$ at $r$ very general points in $X$.
\end{abstract}
\section{Introduction}
Seshadri constants are a local measure of positivity of a line bundle
on a projective variety. They arose out of an ampleness criterion of
Seshadri, \cite[Theorem 7.1]{Har70} and were defined by
Demailly, \cite{Dem}. They have become an important area of
research with interesting connections to other areas of mathematics. 
For a comprehensive account of Seshadri constants, see \cite{primer}. 

Let $X$ be a smooth projective variety and let $L$ be a nef line
bundle on $X$. Let $r \ge 1$ be an integer. For $x_1,\ldots,x_r \in
X$, the {\it Seshadri constant} of $L$ at $x_1,\ldots,x_r$ is defined
as follows. 
$$\varepsilon(X,L,x_1,\ldots,x_r):=  \inf\limits_{C \cap \{x_1,\ldots,x_r\}
  \ne \emptyset} \frac{L\cdot C}{\sum\limits_{i=1}^r {\rm
    mult}_{x_i}C}.$$ 


We remark that the infimum above is the same as the infimum taken over
irreducible, reduced curves $C$ such that $C \cap \{x_1,\ldots,x_r\}
\ne \emptyset$. Indeed, we have the inequality
$$\frac{L\cdot
  (C+D)}{\sum_{i=1}^r{\rm mult}_{x_i}(C)+\sum_{i=1}^r {\rm mult}_{x_i}(D)} \ge {\rm min} \left(\frac{L\cdot
  C}{\sum_{i=1}^r{\rm mult}_{x_i}(C)}, \frac{L\cdot
  D}{\sum_{i=1}^r{\rm mult}_{x_i}(D)}\right).$$

The Seshadri criterion for ampleness says that a line bundle $L$ on a
projective variety is ample if and only if $\varepsilon(X,L,x)>0$ for
every $x \in X$. 

The following is a well-known upper bound for Seshadri constants. Let
$n$ be the dimension of $X$. Then for
any $x_1,\ldots,x_r \in X$, 
\begin{eqnarray*}\label{wellknown}
\varepsilon(X,L,x_1,\ldots,x_r) \le
\sqrt[n]{\frac{L^n}{r}}.
\end{eqnarray*}

Next one defines $$\varepsilon(X,L,r) : = \max\limits_{x_1,\ldots,x_r \in X}
\varepsilon(X,L,x_1,\ldots,x_r).$$ 

It is known that $\varepsilon(X,L,r)$ is
attained at a very general set of points $x_1,\ldots,x_r \in X$; see
\cite{Ogu}. Here {\it very general} means that $(x_1,\ldots,x_r)$ is
outside a countable union of proper Zariski closed sets in $X^r  = X
\times X \times \ldots \times X$. 

In this
paper we will be concerned with lower bounds for $\varepsilon(X,L,r)$
when $X$ is a smooth projective surface with
Picard number 1, $L$ is the ample generator
of the 
N\'eron-Severi
group of
$X$ (see definitions below) and $r\ge 2$ is an integer. 

There has been extensive work on lower bounds for
$\varepsilon(X,L,1)$.

Let $X$ be an arbitrary smooth complex projective surface and let $L$
be an ample line bundle on $X$. It is easy to see that if $L$ is very
ample then $\varepsilon(L,x) \ge 1$ for any point $x \in X$. On the other hand,
if $L$ is ample, but not very ample, it is possible that
$\varepsilon(L,x) < 1$. In fact, Seshadri constants can be arbitrarily
small. Miranda has shown that given any {\it rational} number $\varepsilon > 0$, there exist a
surface $X$, an ample line bundle $L$ and a point $x \in X$ such that 
$\varepsilon(X,L,x) = \varepsilon$; see \cite[Example 5.2.1]{L}. In
Miranda's construction, the surface $X$ typically has a large Picard
number. It is possible that $\varepsilon(X,L,x) <
1$ even on a surface with Picard number 1; see \cite[Example 1.2]{BS}.

However, in an important paper \cite{EL}, 
Ein and Lazarfeld proved that
$\varepsilon(X,L,1) \ge 1$, if $L$ is ample. 
In fact, they prove the following theorem.

\begin{theorem}[Ein-Lazarsfeld]  
Let $X$ be a smooth projective surface and let $L$ be an ample line
bundle on $X$. Then 
$\varepsilon(X,L,x) \ge 1$ for all except possibly countably many points $x \in X$. Further,
if $L^2 > 1$, the set of exceptional points is actually
finite. 
\end{theorem}

There are many other results calculating $\varepsilon(X,L,1)$ or giving
bounds for it. 
We will just mention a few that are of relevance to
us here. Bauer \cite{Bau} considered several cases of surfaces including
abelian surfaces of Picard number 1. 
Szemberg, in \cite{Sze} and \cite{Sze08},
dealt with surfaces of Picard number 1. Szemberg has a
conjecture for  a lower bound for $\varepsilon(X,L,1)$ for surfaces
with Picard number 1; see \cite[Conjecture]{Sze}. \cite{FSST}
gives new lower bounds for $\varepsilon(X,L,1)$ for an arbitrary surface
and an ample line bundle $L$. In the process, it verifies 
the conjecture
in \cite{Sze} in many cases.

Contrary to the case of $\varepsilon(X,L,1)$, 
there are not too many lower bounds for 
$\varepsilon(X,L,r)$ in the literature when $r \ge 2$.
We mention some results in this case. 

 K\"{u}chle  \cite{K} considers multi-point Seshadri
constants on an arbitrary projective variety. Syzdek and Szemberg  
\cite{SS} prove a lower bound for multi-point Seshadri
constants on 
any surface. 

For an arbitrary surface $X$ and a nef and big line bundle $L$ on $X$,
\cite[Theorem 1.2.1]{HR08} gives good lower bounds for
$\varepsilon(X,L,r)$ in terms of the degree of the least degree curve
passing through $r$ general points with a given multiplicity
$m$. Using this theorem, \cite[Corollary 1.2.2]{HR08} has a bound for
$\varepsilon(X,L,r)$ when $L$ is the ample generator the N\'eron-Severi group of $X$.

For an arbitrary surface $X$, when $L$ is very ample, \cite[Theorem I.1]{Har} gives lower bounds for
$\varepsilon(X,L,r)$. 

When $X$ has Picard number 1 and $L$ is the ample
generator of the N\'eron-Severi group, \cite[Theorem 3.2]{Sze} also gives
lower bounds for $\varepsilon(X,L,r)$. \cite[Section 8]{Bau} considers
multi-point Seshadri constants for abelian surfaces. 

Another method of obtaining lower bounds for $\varepsilon(X,L,r)$ comes
from the following bound (see \cite{Bir,Roe,RR}):
\begin{eqnarray}\label{biran}
\varepsilon(X,L,r) \ge
\varepsilon(X,L,1)\varepsilon(\proj^2,\str_{\proj^2}(1),r).
\end{eqnarray}

The well-known Nagata Conjecture is a statement about
$\varepsilon(\proj^2,\str_{\proj^2}(1),r)$. It says that, for $r \ge 10$, 
$\varepsilon(\proj^2, \str_{\proj^2}(1),r) = \frac{1}{\sqrt{r}}.$
This was proved by Nagata when $r$ is a square and is open in all other
cases. But there are several results which give good bounds
close to the bound expected by the Nagata Conjecture (see \cite{Har,HR}). 
Using known bounds for $\varepsilon(X,L,1)$, some of which were mentioned
above, and bounds on $\varepsilon(\proj^2,\str_{\proj^2}(1),r)$, one can get bounds 
on $\varepsilon(X,L,r)$ by \eqref{biran}.

Let $X$ be a smooth projective surface over $\complex$. 
The {\it Picard group} of $X$, denoted $\rm{Pic}(X)$, is the group of
isomorphism classes of line bundles on $X$. The {\it divisor
class group} of $X$, denoted $\rm{Div}(X)$, is the group of divisors on
$X$ modulo linear equivalence. Then $\rm{Pic}(X)$ and $\rm{Div}(X)$ are
isomorphic as abelian groups. The {\it N\'eron-Severi} group of $X$ is defined as $$\rm{NS}(X) :=
\rm{Pic}(X)/\rm{Pic}^0(X),$$ where $\rm{Pic}^0(X)$ is the subgroup of 
$\rm{Pic}(X)$ consisting of line bundles which are algebraically equivalent to
the trivial line bundle. The N\'eron-Severi group $\rm{NS}(X)$ is a finitely generated abelian
group and the {\it Picard number} of $X$ is defined to be $$\rho(X) := \rm{rank ~
  NS}(X).$$ 

Our main result Theorem \ref{main}
gives lower bounds for $\varepsilon(X,L,r)$ when 
 $X$ has Picard number 1 and $L$ is the ample
generator of $\rm{NS}(X)$. Some examples of surfaces with Picard
number 1 include $\proj^2$, general K3 surfaces, and general hypersurfaces
in $\proj^3$ of degree at least 4. 

Our bounds are not always easily comparable to bounds given previously
in the literature, but in many cases, we get better bounds.
We give some examples comparing our result with 
known bounds.

We work throughout over the complex number field $\complex$. By a {\it
  surface}, we mean a nonsingular projective variety of dimension 2. 
We denote the Seshadri constant $\varepsilon(X,L,r)$ simply by
$\varepsilon(L,r)$ when the surface $X$ is clear from context. 

{\bf Acknowledgements:} We sincerely thank Brian Harbourne and Tomasz Szemberg for carefully
reading this paper and making many useful suggestions. We also thank
the referee for making numerous corrections and suggestions that
improved the paper. 

\section{Main theorem}
The following is the main theorem of this paper. 
\begin{theorem}\label{main}
Let $X$ be a smooth projective surface with Picard number 1. Let $L$ be
the ample generator of $\rm{NS}(X)$. 
Let $r \ge 2$ be an
integer. Then we have the following.
\begin{enumerate}
\item If $(r,L^2) = (2,6)$ then $\varepsilon(L,r) \ge
  \frac{3}{2}$. If the equality holds, $\varepsilon(L,2)$ is achieved by a curve $C
  \in |L|$ which passes through two very general points with multiplicity
  two each. \label{3}
\item If $(r,L^2) \ne (2,6)$ then one of the
  following holds: 
\begin{enumerate}
\item $\varepsilon(L,r) \ge
\sqrt{\frac{r+2}{r+3}}\sqrt{\frac{L^2}{r}}$, or \label{1} 
\item if  $\varepsilon(L,r) <
\sqrt{\frac{r+2}{r+3}}\sqrt{\frac{L^2}{r}}$ then, for some $d \ge 1$, there exists an
irreducible, reduced curve $C\in |dL|$ which passes through $s\le r$
very general points $x_1,\ldots,x_s$ with multiplicity one each, $s-1 \le C^2$,
and 
$\varepsilon(L,r) =  \frac{C\cdot L}{s} = \frac{dL^2}{s}$.  \label{2}
\end{enumerate}
\end{enumerate}
\end{theorem}
\begin{proof}
If $\varepsilon(L,r)$ is optimal $\Big{(}$namely, equal to
$\sqrt{\frac{L^2}{r}}\Big{)}$, there is nothing to prove. So we assume that 
$\varepsilon(L,r) < \sqrt{\frac{L^2}{r}}$. In this case, 
there is in fact
an irreducible and reducible  curve $C$ which computes
$\varepsilon(L,r)$. See \cite[Proposition 4.5]{Sze01}, \cite[Lemma
2.1.2]{HR08}, or \cite[Proposition 1.1]{BS}.

Since the Picard number of $X$ is 1, $C$ is algebraically equivalent,
and hence numerically equivalent, to 
$dL$ for some $d \ge 1$. Let $k=L^2$. Then $L\cdot C = kd$ and $C^2
= kd^2$. 

Let
$x_1,\ldots,x_r$ be very general points of $X$ so that $\varepsilon(L,r) = \varepsilon(L,x_1,\ldots,x_r)$. 
Denote the multiplicity of $C$ at $x_i$ to be $m_i$. 
We rearrange the points so that $m_1
\ge m_2 \ge \ldots \ge m_r$. 
Let $s \in \{1,\ldots,r\}$ be such that 
$m_s > 0$ and $m_{s+1} = \ldots = m_r = 0$.

Since the points $x_1,\ldots,x_r$ are very general, there is a
non-trivial one-parameter family of irreducible and reduced curves 
$\{C_t\}_{t\in T}$ parametrized by some smooth curve $T$ and containing
points $x_{1,t},\ldots,x_{r,t} \in C_t$ with mult$_{x_{i,t}}(C_t) \ge
  m_i$ for all $1\le i \le r$ and $t \in T$. 

By a result of Ein-Lazarsfeld \cite{EL} and Xu \cite[Lemma 1]{X1}, we have 
\begin{eqnarray}\label{el}
d^2k = C^2 \ge
m_1^2+\ldots+m_s^2-m_s.
\end{eqnarray}

First assume that $m_1 = 1$. 
Then $m_i = 1$ for all $i=1,\ldots,s$  and 
$m_{s+1}=\ldots=m_r=0$. By \eqref{el}, we have $s-1 \le C^2$. 
Since $C$ computes the Seshadri constant, $\varepsilon(L,r) = \frac{C
  \cdot L}{s} = \frac{dk}{s}$. So we are in case \eqref{2}.

We assume now that $m_1 \ge 2$.

The desired inequality in case \eqref{1} is 
$\frac{L\cdot C}{\sum_{i=1}^r m_i} = \frac{dk}{\sum_{i=1}^rm_i} \ge \sqrt{\frac{r+2}{r(r+3)}}\sqrt{k}$.
It is equivalent to 
\begin{eqnarray}\label{desired1}
d^2k \ge \frac{r+2}{r(r+3)} \left(\sum_{i=1}^r
  m_i\right)^2.
\end{eqnarray}

Note that it suffices to prove 
\begin{eqnarray}\label{desired}
d^2k \ge \frac{s+2}{s(s+3)} \left(\sum_{i=1}^r
  m_i\right)^2.
\end{eqnarray}
Indeed, it is clear that $\frac{s+2}{s(s+3)} \ge \frac{r+2}{r(r+3)}$
for $s \le r$.

We use the following inequality which is a special case of 
\cite[Lemma 2.3]{Han}. Suppose $m_1 \ge 2$, $s \ge 3$ or if $s=2$ then
$(m_1,m_2) \ne (2,2)$. 
\begin{eqnarray}\label{ineq}
\frac{(s+3)s}{s+2}\left(\sum_{i=1}^s m_i^2 -m_s\right) \ge
  \left(\sum_{i=1}^s m_i\right)^2.
\end{eqnarray}

If $s \ge 3$ then \eqref{ineq} is
applicable and together with \eqref{el}, it easily gives \eqref{desired}.  
Suppose next that $s=2$. If $(m_1,m_2) \ne (2,2)$ we
again use \eqref{ineq} to obtain \eqref{desired}. 

If $(m_1,m_2) = (2,2)$, the right hand side of \eqref{desired} is
$\frac{32}{5}$. By \eqref{el}, $d^2k \ge 6$. So \eqref{desired} fails
only if $d=1$ and $k=6$. So we have case \eqref{3} of the theorem.

Finally, if $s=1$ then we have $d^2k \ge m_1^2-m_1$ by
\eqref{el}. Since $r \ge 2$, we have $m_1^2-m_1 \ge \frac{r+2}{r(r+3)}m_1^2$ if $m_1\ge 2$. So
\eqref{desired1} holds. If $m_1=1$ then we are in case \eqref{2}.

This completes the proof.
\end{proof}

\begin{remark}\label{nagata-biran-szemberg}
Let $X$ be any surface and $L$ an ample line bundle on $X$. The
Nagata-Biran-Szemberg Conjecture 
predicts that $\varepsilon(X,L,r) = \sqrt{\frac{L^2}{r}}$ for $r\ge r_0$,
for some $r_0$ depending on $X$ and $L$ 
(see \cite{Sze01}). Our bound in Theorem \ref{main} shows that on a
surface $X$ with Picard number 1, $\varepsilon(X,L,r)$ is arbitrarily
close to the optimal value of $\sqrt{\frac{L^2}{r}}$ as $r$ tends to
$\infty$, as we explain below. 

If Case \eqref{2} of Theorem \ref{main} holds then $\varepsilon(X,L,r)
= \frac{dk}{s}$ where $k=L^2$ and $C$ is a curve numerically
equivalent to $dL$ which passes through $s \le r$ general points. Since we always
have $\varepsilon(X,L,r) \le \sqrt{\frac{k}{r}}$, we get $d^2 \le
  \frac{s^2}{rk}$. On the other hand, for such a curve $C$ to exist,
  we must have $s < h^0(X,dL)$. By the Riemann-Roch theorem applied
  asymptotically, 
the dimension of the space of global sections of $dL$ is bounded by $\frac{d^2L^2}{2} =
  \frac{d^2k}{2}$. Hence $s \le  \frac{d^2k}{2}$. Putting the above
  two inequalities together, we get $d^2 \le s \frac{s}{rk} \le
  \frac{d^2k}{2}\frac{s}{rk}$. This in turn gives $2r \le s$, which is a
  contradiction. 

Hence Case \eqref{2} of Theorem \ref{main} does {\em
  not} hold for large $r$. This means that Case \eqref{1} does. So we
conclude that $\varepsilon(X,L,r)$ is arbitrarily
close to the optimal value of $\sqrt{\frac{L^2}{r}}$ as $r$ tends to
infinity.

\end{remark}

\begin{remark}\label{HR08}
Let $X$ be any surface and $L$ a nef and big line
bundle. In \cite[Theorem 1.2.1]{HR08}  Harbourne and Ro\'e give lower bounds for
$\varepsilon(X,L,r)$. When $X$ has Picard number 1 and $L$ is an ample
line bundle on $X$, \cite[Corollary 1.2.2]{HR08} gives
similar bounds.  
The bounds stated in \cite{HR08} are implicit and therefore difficult
to compare with our bounds. In all situations when the statements in
\cite{HR08} can be made effective, the bounds in \cite{HR08} turn out
to be better than our bounds in Theorem \ref{main}.

As an application, Harbourne and Ro\'e give
good lower bounds when $X = \proj^2$; see \cite[Corollary 1.2.3]{HR08}. 
In order to use these bounds, one needs information about the smallest
degree curves passing through $r$ general points with a given
multiplicity $m$. For specific classes of surfaces, it may be possible to
obtain this information and in turn get good lower bounds for
$\varepsilon(X,L,r)$. 
\end{remark}

\begin{remark}
Let $X$ be a surface with Picard number 1 and let $L$ be the ample 
generator of $\rm{NS}(X)$. Let $r \ge 1$. 
Szemberg \cite[Theorem 3.2]{Sze} obtains the following bound: 
\begin{eqnarray}\label{szemberg}
  \varepsilon(L,r) \ge \left\lfloor \sqrt{\frac{L^2}{r}}\right\rfloor.
\end{eqnarray}

We compare our result Theorem \ref{main} with  
this bound.

If  $\frac{L^2}{r}$ is the square of an integer, \eqref{szemberg} shows that the Seshadri
 constant is optimal, namely equal to $\sqrt{\frac{L^2}{r}}$, while
our bound is sub-optimal. So the bound in
 \eqref{szemberg} is always better in this situation. 

On the other hand, our bound is always better if $L^2 < r$, because
the bound in \eqref{szemberg} is 0 in this case. 

If $L^2$ is large compared to $r$, the bound in \eqref{szemberg} is
better. More precisely, for a fixed $r$, there is some $N_r$ such that 
if $L^2 \ge N_r$, \eqref{szemberg} uniformly gives a better bound. 
However, if $L^2 < N_r$, our bound is better in some cases and
\eqref{szemberg} is better in some cases.

We give an example to illustrate this. Before considering the example,
we note that given any integer $k\ge 4$, a general hypersurface $X$ of
degree $k$ in $\proj^3$ has Picard number 1. In fact, if $L = \str_X(1)$, the Picard group
of $X$ is isomorphic to $\ints \cdot L$, by
Noether-Lefschetz. Further, $L^2 = k$.

Let $r=10$. For $L^2 = 150$, the 
the bound in \cite{Sze} is 3, while our bound
is 3.72.   
For $L^2 = 1050$, the bound in \cite{Sze} is 10, while ours is 9.84. 
If $L^2 = 2500$, then the bound in \cite{Sze} is 15, and our bound is
15.19. Starting at about $L^2=5000$, the bound in \cite{Sze} is
uniformly better than ours.  
\end{remark}
\begin{remark}\label{harbourne}
Let $X$ be a surface and let $L$ be a very ample line bundle. Let $k =
L^2$. Let $\varepsilon_{r,k}$ be the maximum element in the following
set: 
$$\Bigg\{\frac{\lfloor d\sqrt{rk}\rfloor}{dr}\Bigg|1 \le d \le
\sqrt{\frac{r}{k}}\Bigg\} \cup \Bigg\{ \frac{1}{\lceil
  \sqrt{\frac{r}{k}} \rceil} \Bigg\} \cup \Bigg\{ \frac{dk}{\lceil
  d\sqrt{rk}\rceil}  \Bigg| 1 \le d \le \sqrt{\frac{r}{k}}\Bigg\}.$$

Harbourne \cite[Theorem I.1]{Har} proves that $\varepsilon(L,r) \ge
\varepsilon_{r,k}$, unless $k\le r$ and $rk$ is a square, in which case 
$\sqrt{\frac{k}{r}} = \varepsilon_{r,k}$ and $\varepsilon(L,r) \ge
\sqrt{\frac{k}{r}}-\delta$, for every positive rational number $\delta$.

Suppose now that $X$ has Picard number 1 and let $L$ be a very ample line
bundle on $X$.

When $r < L^2$ our bound is better than \cite[Theorem I.1]{Har}. 
If $r \ge L^2$, our result is not always uniformly comparable with this result.
The
two bounds are in general very close to each other. 
We give an example to illustrate this. Let $r=10$. Then for $L^2  =6$
our bound is 0.744 and the bound in \cite{Har} is 0.75. On the other
hand, for $L^2=7$, our bound is 0.803 while the bound in \cite{Har} is
0.8. 

It should be noted however
that \cite[Theorem I.1]{Har} is a general result proved for a very
ample line bundle on \emph{any} surface. 
\end{remark}

\begin{remark}\label{from-biran}
We compare our bounds with the bounds that can be derived from
\eqref{biran}. 

Let $X$ be a surface with Picard number 1 and let $L$ be the ample
generator of $\rm{NS}(X)$. \cite[Conjecture]{Sze} says that 
$\varepsilon(X,L,1) \ge \frac{p_o}{q_o}k$ where $k = L^2$ and 
$(p_o,q_0)$ is a primitive solution to Pell's
equation: $q^2-kp^2=1$. \cite{FSST} verifies this conjecture whenever
$k$ is of the form $n^2-1$ or $n^2+1$ for a positive integer $n$. In all
cases (not just if $X$ has Picard number one) it is proved that
$\varepsilon(X,L,1) \ge  \frac{p_o}{q_o}k$ or $\varepsilon(X,L,1)$ belongs to a
finite set of exceptional values. 

Let $r=101$, $L^2=35$. Then our bound is 0.5858. Assuming
$\varepsilon(\proj^2,\str_{\proj^2}(1),r) = \sqrt{\frac{1}{r}}$ (which is the
value predicted by the Nagata Conjecture)  and using the
bound $\varepsilon(X,L,1) \ge \frac{35}{6}$ in \cite[Theorem 4.1]{FSST},
\eqref{biran} gives us
$\varepsilon(X,L,r) \ge 0.5804$. If $L$ is very ample, \cite[Theorem
I.1]{Har} gives
the bound $0.5833$. Note that the optimal value of $\varepsilon(X,L,r)$ is 
$\sqrt{\frac{35}{101}} = 0.5886$ in this case. 

Bauer \cite[Theorem 6.1]{Bau} explicitly calculates the Seshadri constants
$\varepsilon(X,L,1)$, where $X$ is an abelian surface with Picard number 1, in terms of the primitive solution to
Pell's equation. These values, together with bounds on
$\varepsilon(\proj^2,\str_{\proj^2}(1),r)$ can be used (via
\eqref{biran}) to obtain bounds for $\varepsilon(X,L,r)$. These bounds are
sometimes better than our bound, but not always. 

\end{remark}

\subsection{Examples}

\begin{example}\label{p2}

Let $(X,L) = (\proj^2,\str_{\proj^2}(1))$.

The values of $\varepsilon = \varepsilon(\proj^2,\str_{\proj^2}(1,r))$ for $2 \le r \le
9$ are given below: 

\underline{$r=2,3,4$} : Here  $\varepsilon = 1/2$. In all three cases the Seshadri constant
is achieved by a line through two points. Note that the bound expected
in Case \eqref{1} in Theorem \ref{main} is 0.63, 0.52 and 0.4 when
$r=2,3,4$ respectively. So when $r=2$ or $r=3$, Case \eqref{2} in
Theorem \ref{main} is realized.  

\underline{$r=5$}: In this case $\varepsilon = 2/5$, achieved by a conic
through five general points. The bound predicted by Case
\eqref{1} in Theorem \ref{main} is 0.41 which is more than 2/5. 
So this is also an instance where Case \eqref{2} in Theorem \ref{main}
is achieved. 

\underline{$r=6,7,8,9$}: In these cases $\varepsilon = 2/5,3/8,6/17,3$
respectively. In each case, the bound of Case \eqref{1} of Theorem
\ref{main} holds. 

For $r\ge 10$, the Nagata Conjecture says that
$\varepsilon(\proj^2,\str_{\proj^2}(1,r)) = \frac{1}{\sqrt{r}}$.

\end{example}

\begin{example}\label{k3}  Let $X$ be a  general K3 surface. It is well-known that
  Picard number is 1 for such surfaces. Let $L$ be the ample generator
  of $\rm{NS}(X)$ with $k = L^2$. Let $r \ge 3$. 
We will show in this case that 
$\varepsilon(X,L,r) \ge \sqrt{\frac{r+2}{r+3}}\sqrt{\frac{k}{r}}$. 

If not, by Theorem \ref{main}, for some $d \ge 1$, there
exists a curve $C\in |dL|$ passing through $s\le r$ very general
points with multiplicity one at each point and such that
$\varepsilon(X,L,r) = \frac{L \cdot C}{s}$. 

We first suppose that $C^2 = d^2k \ge s$. Then 
$\varepsilon(X,L,r) = \frac{dk}{s}\ge \sqrt{\frac{k}{r}}$.
But this violates the assumption that 
$\varepsilon(X,L,r) < \sqrt{\frac{r+2}{r+3}}\sqrt{\frac{k}{r}}$

Now we compute $h^0(dL)$ and show that $C^2 \ge s$ always holds.
By the Kodaira vanishing, $h^1(dL) = h^2(dL) = 0$. It follows from the 
Riemann-Roch theorem that  $h^0(dL) = \frac{d^2k}{2}+2$
for $d \ge 1$. 
Since the linear system $|dL|$ contains
a curve through $s$ general points, its dimension is at least $s$. On
the other hand, if its dimension is more than $s$, then $|dL|$
contains a curve $D$ through at least $s+1$ points. Then by
inequality \eqref{el}, $d^2k = C^2 = D^2 \ge s$. In this case, we are done by the
argument in the previous paragraph. So suppose that 
$h^0(dL) = \frac{d^2k}{2}+2 = s+1$. Hence 
$C^2 =  d^2k = 2s-2$. If $s\ge 2$, then $C^2 \ge s$. If $s=1$, then 
anyway $C^2 = d^2k \ge 1$. So in both cases, we are done by the 
argument in the previous paragraph.

\end{example}

\begin{remark}
Another class of surfaces which can have Picard number 1 are abelian
surfaces. Bauer \cite{Bau} studies both single and multi-point Seshadri
constants on such surfaces generally (not just when the Picard number
is 1). In \cite[Section 8]{Bau}, lower bounds for multi-point Seshadri
constants are obtained for {\it arbitrary} points $x_1,\ldots,x_r$. 
These bounds are naturally smaller than our bounds for 
$\varepsilon(L,r)$, which is the Seshadri constant at $r$ {\it general}
points.  
\end{remark}

\begin{remark}
We do not know of any example where Case \eqref{3} of Theorem
\ref{main} is realized. 
\end{remark}

Though the Nagata-Biran-Szemberg Conjecture (see Remark
\ref{nagata-biran-szemberg}) predicts that for a fixed polarized
surface $(X,L)$, the Seshadri constants $\varepsilon(X,L,r)$ are
optimal for large enough $r$, we can ask the following question.

\begin{question}\label{question}
Given an integer $r \ge 1$, is there a pair $(X,L)$ where $X$ is a surface
with Picard number 1 and
$L$ is the ample generator of $\rm{NS}(X)$ such that $\varepsilon(X,L,r)$ is sub-optimal,
that is, $\varepsilon(X,L,r) < \sqrt{\frac{L^2}{r}}$? 
\end{question}
For $r=1$, an affirmative answer is given by Bauer \cite[Theorem
6.1]{Bau}. This result shows that for a polarized abelian surface $(X,L)$
of type $(1,d)$, the Seshadri constant $\varepsilon(X,L,1)$ is
sub-optimal if $L^2 = 2d$ is not a square. For $r=2,3,5,6,7$ and $8$, $(X,L)=(\proj^2,\str_{\proj^2}(1))$
(see Example \ref{p2}) gives an affirmative answer. For other values of $r$,
we do not know an answer to Question \ref{question}. 

\bibliographystyle{plain}

\begin{thebibliography}{100}
\bibitem{Bau}  Bauer, Thomas, {\it Seshadri constants on
    algebraic surfaces}, 
Math. Ann. 313 (1999), no. 3, 547-583. 
\bibitem{primer} Bauer, Thomas; Di Rocco, Sandra; Harbourne, Brian;
  Kapustka, Micha\l{}; Knutsen, Andreas; Syzdek, Wioletta; Szemberg,
  Tomasz, {\it A primer on Seshadri constants}, 
Interactions of classical and numerical algebraic geometry, 33-70, Contemp. Math., 496, Amer. Math. Soc., Providence, RI, 2009.
\bibitem{BS} Bauer, Thomas; Szemberg, Tomasz,
{\it Seshadri constants on surfaces of general type},
Manuscripta Math. 126 (2008), no. 2, 167-175. 
\bibitem{Bir} Biran, Paul,
{\it Constructing new ample divisors out of old ones},
Duke Math. J. 98 (1999), no. 1, 113-135. 
\bibitem{Dem} Demailly, Jean-Pierre, {\it  Singular Hermitian
    metrics on positive line bundles}, 
Complex algebraic varieties (Bayreuth, 1990), 87-104, 
Lecture Notes in Math., 1507, Springer, Berlin, 1992. 
\bibitem{EL} Ein, Lawrence; Lazarsfeld, Robert, {\it Seshadri constants on smooth
    surfaces}, Journ\'{e}es de G\'{e}om\'{e}trie Alg\'{e}brique d'Orsay (Orsay,
  1992), Ast\'{e}risque No. 218 (1993), 177-186.
\bibitem{FSST}  Farnik, \L{}ucja; Szemberg, Tomasz;  Szpond, Justyna;
  Tutaj-Gasi\'{n}ska, Halszka, {\it Restrictions on Seshadri constants on
    surfaces}, arXiv:1602.08984v1, to appear in Taiwanese J. Math.
\bibitem{Han} Hanumanthu, Krishna, {\it Positivity of line bundles on
    general blow ups of $\proj^2$}, J. Algebra 461 (2016), 65-86.
\bibitem{Har} Harbourne, Brian, {\it Seshadri constants and very
    ample divisors on algebraic surfaces}, J. Reine Angew. Math. 559 (2003), 115-122.
\bibitem{HR08} Harbourne, Brian; Ro\'{e}, Joaquim, {\it Discrete
    behavior of Seshadri constants on surfaces}, 
J. Pure Appl. Algebra 212 (2008), no. 3, 616-627. 
\bibitem{HR} Harbourne, Brian; Ro\'{e}, Joaquim,
{\it Computing multi-point Seshadri constants on $\proj^2$},
Bull. Belg. Math. Soc. Simon Stevin 16 (2009), no. 5, Linear systems and subschemes, 887-906. 
\bibitem{Har70} Hartshorne, Robin, {\it 
Ample subvarieties of algebraic varieties},
Notes written in collaboration with C. Musili. Lecture Notes in
Mathematics, Vol. 156 
Springer-Verlag, Berlin-New York 1970.
\bibitem{K} K\"{u}chle, Oliver, {\it Multiple point Seshadri constants and the dimension
    of adjoint linear series}, Ann. Inst. Fourier (Grenoble) 46 (1996), no. 1, 63-71. 
\bibitem{L} Lazarsfeld, Robert, 
{\it Positivity in algebraic geometry I 
Classical setting: line bundles and linear series},  Ergebnisse der Mathematik und ihrer Grenzgebiete. 3. Folge. A Series of Modern Surveys in Mathematics 48, Springer-Verlag, Berlin, 2004.
\bibitem{Ogu} Oguiso, Keiji, {\it 
Seshadri constants in a family of surfaces}, 
Math. Ann. 323 (2002), no. 4, 625-631.
\bibitem{Roe} Ro\'{e}, Joaquim,
{\it A relation between one-point and multi-point Seshadri constants},
J. Algebra 274 (2004), no. 2, 643-651. 
\bibitem{RR} Ro\'{e}, Joaquim; Ross, Julius,
{\it An inequality between multipoint Seshadri constants},
Geom. Dedicata 140 (2009), 175-181. 
\bibitem{SS} Syzdek, Wioletta; Szemberg, Tomasz, 
{\it Seshadri fibrations of algebraic surfaces},
Math. Nachr. 283 (2010), no. 6, 902-908.  
\bibitem{Sze01} Szemberg, Tomasz, {\it Global and local
    positivity of line bundles}, Habilitation 2001. 
\bibitem{Sze} Szemberg, Tomasz,
{\it Bounds on Seshadri constants on surfaces with Picard number 1},
Comm. Algebra 40 (2012), no. 7, 2477-2484. 
\bibitem{Sze08}  Szemberg, Tomasz, {\it An effective and sharp lower bound on Seshadri constants on surfaces with Picard number 1}, J. Algebra 319 (2008), no. 8, 3112-3119.
\bibitem{X1} Xu, Geng, {\it Curves in $\proj^2$ and symplectic packings}, Math. Ann. 299 (1994), 609-613.

\end{thebibliography}

\end{document}